\documentclass{amsart}
\usepackage{graphicx}
\usepackage{graphics}
\usepackage{amsmath,amssymb}
\usepackage{amscd,enumerate}
\usepackage{latexsym}
\usepackage{amssymb}
\usepackage[all]{xy}
\usepackage{psfrag}
\usepackage{amsthm}
\usepackage{color}
\xyoption{matrix} \xyoption{arrow}
\usepackage[all]{xy}
\input xy
\xyoption{all}

\renewcommand{\mod}{\operatorname{mod}\nolimits}
\newcommand{\Mod}{\operatorname{Mod}\nolimits}

\newcommand{\add}{\operatorname{add}\nolimits}

\newcommand{\Hom}{\operatorname{Hom}\nolimits}

\newcommand{\End}{\operatorname{End}\nolimits}

\newcommand{\op}{{\operatorname{op}}}

\newcommand{\C}{\operatorname{\mathcal C}\nolimits}

\newcommand{\F}{\operatorname{\mathcal F}\nolimits}

\renewcommand{\P}{\operatorname{\mathcal P}\nolimits}
\newcommand{\K}{\operatorname{\mathcal K}\nolimits}
\newcommand{\T}{\operatorname{\mathcal T}\nolimits}

\newcommand{\M}{\operatorname{\mathcal M}\nolimits}
\newcommand{\N}{\operatorname{\mathcal N}\nolimits}
\newcommand{\MM}{\operatorname{\overline{\mathcal M}*}\nolimits}
\newcommand{\NN}{\operatorname{\overline{\mathcal N}*}\nolimits}
\newcommand{\TT}{\operatorname{\overline{\mathcal T}*}\nolimits}

\newcommand{\G}{\Gamma}

\newcommand{\val}{\operatorname{val}\nolimits}

\newtheorem{lemma}{Lemma}[section]
\newtheorem{proposition}[lemma]{Proposition}
\newtheorem{corollary}[lemma]{Corollary}
\newtheorem{theorem}[lemma]{Theorem}
\theoremstyle{definition}
\newtheorem{remark}[lemma]{Remark}
\newtheorem{definition}[lemma]{Definition}

\begin{document}

\topmargin 0cm
\oddsidemargin 0.5cm
\evensidemargin 0.5cm
\baselineskip=14pt

\title[The geometry of Brauer graph algebras]{The geometry of Brauer graph algebras and cluster mutations}

\author[Marsh]{Robert J. Marsh}
\address{Robert J. Marsh\\
School of Mathematics \\
University of Leeds\\
Leeds LS2 9JT\\
United Kingdom}
\email{marsh@maths.leeds.ac.uk}

\author[Schroll]{Sibylle Schroll}
\address{Sibylle Schroll\\
Department of Mathematics \\
University of Leicester \\
University Road \\
Leicester LE1 7RH \\
United Kingdom}
\email{ss489@le.ac.uk}

\subjclass[2010]{Primary  16G10, 16G20, 16E35; Secondary 13F60, 14J10}
\keywords{Special biserial algebras, Brauer graph algebras, tilting mutation, derived equivalence, ribbon graphs, marked surfaces, triangulations, cluster}

\begin{abstract}
In this paper we establish a connection between ribbon graphs and Brauer graphs.
As a result, we show that a compact oriented surface with marked points
gives rise to a unique Brauer graph algebra up to derived equivalence.
In the case of a disc with marked points we show that a dual construction in
terms of dual graphs exists. The rotation of a diagonal in an
$m$-angulation gives rise to a Whitehead move in the dual graph, and we
explicitly construct a tilting complex on the related Brauer graph algebras
reflecting this geometrical move.
\end{abstract}

\thanks{This work was supported by the Engineering and Physical Sciences
Research Council [grant number EP/G007490/1] and by the Leverhulme Trust
through an Early Career Fellowship for the second named author.}

\date{6 August 2014}
\maketitle

\section*{Introduction}

Recently there has been widespread renewed interest --- spurred on by the many interesting developments in cluster algebras --- in tilting phenomena. In particular, the geometrical interpretation of cluster-tilting objects in cluster categories as triangulations
suggests the consideration of similar models in relation to other algebras. One such area of interest is the class of finite-dimensional self-injective algebras.

Finite dimensional self-injective algebras encompass a large variety of algebras, such as the group algebras of finite groups, Hecke algebras and special biserial algebras. The last family is particularly interesting, as large parts of the representation theory and (co)homology are well-studied and understood.
The class of symmetric special biserial algebras coincides with the class of Brauer graph algebras~\cite{schroll}.
Brauer graph algebras
can be described in two important ways, firstly as a quotient of the path algebra of a quiver, and secondly via a graph, the Brauer graph, with a circular ordering on the edges around each vertex, together
with some special (exceptional) vertices with associated multiplicities. This
combinatorial description makes them particularly amenable to calculations and
combinatorial studies.

We observe that the definition of a Brauer graph without exceptional
vertices is essentially that of a ribbon graph (or fat graph).
Such a graph has a natural filling embedding into an oriented surface (see~\cite{Labourie}) giving rise to the cyclic ordering of the edges incident with each vertex. Conversely, such an embedding
gives rise to a Brauer graph without exceptional vertices.

In particular, Kauer~\cite{Kauer98} shows that if a certain local move
(which we call a \emph{Kauer move}) is applied to the graph of a Brauer graph
algebra, the resulting Brauer graph algebra is derived equivalent to the initial
algebra. We show that if a Brauer graph arises from a triangulation
of a compact oriented marked surface (regarded as the embedding of a graph),
then the Kauer move coincides with the flip of the triangulation which
corresponds to mutation in the cluster algebra associated to the surface~\cite{FominShapiroThurston08}. It follows that we can
associate a Brauer graph algebra to a compact oriented marked surface
(by choosing a triangulation) which is unique up to derived equivalence.
We also give a version of this for the case of a surface with boundary.

In the case of a disc with marked points on its boundary we consider not only triangulations but $m$-angulations. For such an $m$-angulation, regarded as a Brauer graph without exceptional
vertex, the Kauer move again coincides with the mutation
of the $m$-angulation~\cite{HughThomas07} (see also~\cite[\S11]{BuanThomas09})
in the cluster sense. We may also consider the dual graph of the $m$-angulation,
which is an $(m-1)$-ary tree. Mutation of the $m$-angulation induces a move
which coincides with a \emph{nearest neighbour interchange} (in the
sense of~\cite[\S2]{watermansmith73}, which refers to~\cite{MGB73}), or
\emph{Whitehead move}, i.e.\ the contraction and expansion of a given edge.
Note that in the case of a triangulation the effect on the dual graph
corresponds to the associativity rule as described, for example,
in~\cite{fominreading07}.

In the case of a triangulation of a disc, a mutation or flip of a
diagonal corresponds to four different things simultaneously: a mutation of
the associated cluster, a derived equivalence of the Brauer graph algebra
whose graph is the triangulation, a derived equivalence of the Brauer tree
algebra of the dual graph of the triangulation and,
using~\cite{demonetluo,Palu09} (see also~\cite{BaurKingMarsh,JensenKingSu}),
a derived equivalence of endomorphism algebras in Frobenius categories in the
cluster context.

We observe that in the case of Brauer graph algebras the Brauer graph algebra
associated to an $m$-angulation is of tame representation type whereas the one
associated to the dual graph is of finite representation type.

In Section 1 we recall the concept of a Brauer graph algebra and recall the
pivotal result of Kauer on derived equivalences of Brauer graph algebras which
lies at the heart of this paper.
As we expect the audience of this paper to have a more algebraic background we
begin Section 2 with a short recall of the theory of ribbon graphs and
surfaces, and go on to show how up to derived equivalence there is a unique
Brauer graph algebra associated to every compact oriented marked surface
(with or without boundary). In Section 3 we explain how our results are
connected with the theory of cluster algebras and cluster categories.
In Section 4 we consider $m$-angulations of a disc with marked points on its
boundary. We give an explicit two-term tilting complex, compatible with the
geometry, realising the derived equivalence of Brauer tree algebras induced by
the move on the dual graph induced by mutation of a diagonal in the
$m$-angulation. Finally, in Section 5 we give counter-examples to show that the
Brauer graph algebras of the dual graphs associated to two graphs
related by the Kauer move needn't necessarily be derived equivalent in
general, by considering the cases of a triangulation of a sphere and a
punctured disk.

We would like to thank the referee for helpful comments on the first version
of this paper.

\section{Brauer graph algebras}\label{section:notation}

Let $K$ be an algebraically-closed field and $A = KQ /I$ be a finite-dimensional algebra given by the
quiver $Q$ and an admissible ideal $I$.  All modules we consider are left
modules and we denote the category of finite dimensional left $A$-modules by
$A$-$\mod$.  If $ \cdot \stackrel{\alpha}{\longrightarrow} \cdot
\stackrel{\beta}{\longrightarrow} \cdot$ are two consecutive arrows in $KQ$, we
write the corresponding path as $\beta \alpha$.

\subsection{Brauer Graph Algebras}

We will briefly recall the standard setup of Brauer graph algebras.  Details can
be found, for example, in~\cite{Benson98, Kauer98, Roggenkamp96}.

Let $\G$ be a finite graph with at least one edge. Denote by $\G_0$ the set of
vertices of $\G$ and by $\G_1$ the set of edges.  Define a function
$m: \G_0 \to \mathbb{N} \setminus \{0\}$, called the
\emph{multiplicity function} of $\G$. For any graph $\Gamma$,
there is a local embedding of
$\Gamma$ into the plane with the property that an orientation of the plane induces
a cyclic ordering around the vertices of $\Gamma$.

In general, we consider a fixed local embedding of $\G$ and, unless otherwise
stated, the clockwise orientation of the plane and the induced cyclic order of
the edges around every vertex of $\G$.  We call $\G$, equipped with multiplicity
function $m$ and a fixed cyclic ordering a \emph{Brauer graph}.
The valency $\val (X)$ of a vertex $X \in \G_0$ is the
number of edges incident to $X$, where a loop counts twice.
Given a Brauer graph $\G$, the \emph{Brauer graph algebra} $A_{\G}$ is the
algebra $k Q_{\G} / I_{\G}$ where the quiver $Q_{\G}$ is as descibed below. The
ideal $I_{\G}$ is generated by relations $\rho_{\G}$ also described below.

Firstly, if $\G$ is the graph $\xymatrix{ X \ar@{-}[r] & Y }$ with $m(X)=m(Y)=1$
then $Q_{\G} = \xymatrix{ \bullet \ar@(dr,ur)[]_\alpha }$ and
$\rho_{\G} = \{\alpha^2\}$ so that $ A_\G = K[\alpha]/(\alpha^2)$.

In all other cases, the edges $a$ of $\G$ correspond to the vertices $v_a$ of
$Q_{\G}$, and if $b$ is the successor of $a$ in the cyclic ordering around a vertex $X$ then there is an arrow $v_a \stackrel{\alpha}{\longrightarrow}
v_b$ in $Q_{\G}$.

There are three types of relations.  Let $X\in \G_0$, and let $a \in \G_1$ be
an edge incident with $X$.  Let $a = a_1, a_2, \ldots a_{\val (X) }$ be the edges around
$X$ in the cyclic ordering.  Set $C_{a, X} = \alpha_{\val (X)} \ldots \alpha_1$
to be the corresponding cycle of arrows in $K Q_{ \G }$.  Then the relations
$\rho_{\G}$ are defined as follows:

\noindent \emph{Relations of type I}: If the edge $a \in \G_1$ has endpoints
$X$ and $Y$ so that $a$ is not a leaf at $X$ with $m(X)=1$ or at $Y$ with $m(Y)=1$, then we have the
relation $C_{a, X}^{m (X)} - C_{a, Y}^{m (Y)}$ in $\rho_{\G}$.

\noindent \emph{Relations of type II}: If the edge $a \in \G_1$ with endpoints
$X$ and $Y$ is a leaf at the vertex $Y$ with $m(Y)=1$ then $\alpha_1 C_{a, X}^{m (X)} $ is
a relation in $\rho_{\G}$.

\noindent \emph{Relations of type III}: All paths $\alpha\beta$ of length $2$
where $\alpha \beta$ is not a subpath of any cycle $C_{a, X}$ are relations in
$\rho_{\G}$.

\begin{remark}
\begin{enumerate}[(i)]
\item We do not allow truncated edges in the sense of~\cite{Kauer98} here.
\item In \cite{GSS} it is shown that for any given Brauer graph there is a tower
of coverings of Brauer graphs corresponding to a tower of coverings of Brauer graph algebras such that the topmost covering Brauer graph has no loops, no multiple edges and  multiplicity function identical to 1.
\item We do not consider the quantized case of~\cite{GSS}. The Brauer
  graph algebras that we consider here are all symmetric algebras (Recall that a
  finite dimensional $K$-algebra $A$ is symmetric if there exists an isomorphism
  of $A$-$A$-bimodules between $A$ and $\Hom_K(A,K)$).
\end{enumerate}
\end{remark}

We recall part of a result of Kauer~\cite{Kauer98}.

\begin{theorem}\cite[3.5]{Kauer98}\label{thm:KauerMove}
Consider two Brauer graphs, locally embedded in the plane, which are the same
except for a local move as shown in Figure~\ref{fig:KauerMove}.  Then the
corresponding Brauer graph algebras are derived equivalent.
\end{theorem}

\begin{figure}[!h]
\psfragscanon \psfrag{a}{\begin{footnotesize}$a$\end{footnotesize}}
\psfrag{b}{\begin{footnotesize}$b$\end{footnotesize}}
\psfrag{c}{\begin{footnotesize}$c$\end{footnotesize}}
\psfrag{W}{\begin{footnotesize}$W$\end{footnotesize}}
\psfrag{X}{\begin{footnotesize}$X$\end{footnotesize}}
\psfrag{Y}{\begin{footnotesize}$Y$\end{footnotesize}}
\psfrag{Z}{\begin{footnotesize}$Z$\end{footnotesize}}
\begin{center}
\includegraphics[width=10cm]{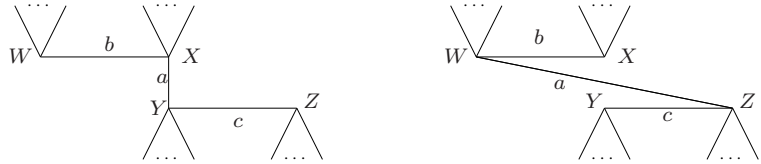} 
\end{center}
\label{fig:KauerMove}
\caption{The Kauer move at the edge $a$.}
\end{figure}
\noindent We call the move appearing in Theorem~\ref{thm:KauerMove} the
\emph{Kauer move} at the edge $a$.

We note that this derived equivalence has later also been
considered in ~\cite{Aihara, Antipov, Dugas} where in some cases it is called
a tilting mutation in analogy with the corresponding mutation in cluster
theory  (See Section 3).


\section{Surfaces and derived equivalences of Brauer graph algebras}
\label{s:surfaces}

The aim of this section is to introduce ribbon graphs and to establish the
fact that a Brauer graph is a ribbon graph with a vertex labelling
corresponding to the values of the multiplicity function of the Brauer graph.
The natural embedding of a ribbon graph into a compact oriented marked surface
gives a natural way to associate a surface to a Brauer graph algebra.
We will also see that, by considering triangulations, we can associate
to every compact oriented marked surface a Brauer graph algebra unique
up to derived equivalence (although we note that this is not the inverse of
the above operation).

First, we recall part of the well-known theory of ribbon graphs and their
embeddings into surfaces, following~\cite{Labourie}.
We will break this section into two parts. In the first part
we consider marked surfaces without boundary and in the second we consider
marked surfaces with boundary.

\subsection{Marked surfaces without boundary}

We begin by giving the definition of a \emph{ribbon graph} or \emph{fat graph}.
An \emph{oriented graph} is a triple $(V,E,\varphi)$, where $V$ is the set of
vertices, $E$ is the set of edges, and $\varphi:E\rightarrow V \times V$ is a
map sending $e\in E$ to the pair $(e^-,e^+)$ of vertices.  We regard $e^-$ as
the start of the arrow $e$ and $e^+$ as the end.  A \emph{graph} is a pair
$(\Gamma,I)$, where $\Gamma=(V,E,\varphi)$ is an oriented graph and
$I:E\rightarrow E$ is a fixed-point-free involution, $e\mapsto \overline{e}$,
satisfying $(\overline{e})^+=e^-$ and $(\overline{e})^-=e^+$ for all $e\in E$.

We call the pair $(e,\overline{e})$ an \emph{undirected edge}. Thus we are
regarding an undirected edge as a pair of directed edges in opposite directions
between the same pair of vertices (which might be equal, in the case of a loop).

The \emph{geometric realisation} $|\Gamma|_I$ of $(\Gamma,I)$ is the topological
space
$|\Gamma|_I=(E\times [0,1])/{\sim},$ where $\sim$ is the equivalence relation
defined by $(e,t)\sim (\overline{e},1-t)$ (for all $t \in [0,1]$), $(e,0)\sim
(f,0)$ if $e^-=f^-$ and $(e,1)\sim (f,1)$ if $e^+=f^+$, for all $e,f\in E$.

A \emph{cyclic} ordering of a finite set $S$ is a bijection $s:S\rightarrow S$
satisfying
$\{s^n(x)\,:\,n\in\mathbb{Z}\}=S,$ for all $x \in S$.  The \emph{predecessor}
of $x\in S$ is $s^{-1}(x)$ and the \emph{successor} of $x$ is $s(x)$.

If $v\in V$, then the \emph{star} of $V$ is
$E_v=\{e\in E\,:\,e^-=v\}.$ A \emph{ribbon graph} (or fat graph) is a graph
together with a cyclic ordering on $s_v$ on $E_v$ for each $v\in V$.  Morphisms
and isomorphisms of ribbon graphs are then defined in the natural way. It is easy
to see that:

\begin{remark}
A ribbon graph is a Brauer graph without exceptional vertices.  In particular,
the way loops work (see Section 1.1 above), ensures that the definitions of
ribbon graphs and Brauer graphs coincide.  Thus a Brauer graph can be regarded
as a ribbon graph with a multiplicity attached to each of its vertices.
\end{remark}

Examples of ribbon graphs include planar graphs and graphs locally embedded in
the plane. An embedding of a graph into an oriented surface induces a
ribbon graph structure on the graph, with cyclic ordering induced from the
embedding around each vertex and the orientation of the surface.

Conversely, every Brauer graph with
multiplicity function identically equal to 1 gives rise to an oriented surface
with boundary:

\begin{lemma} \cite[2.2.4]{Labourie}
Every ribbon graph $\Gamma$ can be embedded in an oriented surface with boundary
in such a way that the cyclic orderings around each of its vertices arise from
the orientation of the surface.
\end{lemma}

The \emph{ribbon surface} $S^{\circ}_{\Gamma}$ of $\Gamma$ is an example of such
a surface. It can be constructed in the following way.  Firstly, a surface is
associated to each vertex and edge of $\Gamma$ as in
Figure~\ref{fig:localsurface}. Then the structure of the graph determines a
gluing of the corresponding surfaces giving an oriented surface together with an
embedding of $\Gamma$.

\begin{figure}
\begin{center}
\includegraphics[width=6cm]{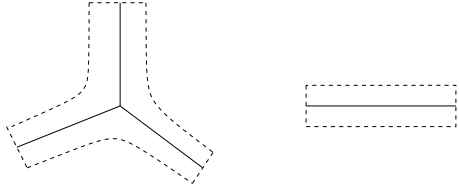}
\end{center}
\caption{Surfaces corresponding to vertices and edges in $\Gamma$.}
\label{fig:localsurface}
\end{figure}

Note that $S^{\circ}_{\Gamma}$ has a number of boundary components. A
\emph{face} of $\Gamma$ is an equivalence class, up to cyclic permutation, of
$n$-tuples $(e_1,e_2,\ldots ,e_n)$ of edges satsifying $e_p^+=e_{p+1}^-$ and
$s_{e_p^+}(\overline{e_p})=e_{p+1}$ for all $p$ with $1\leq p\leq n$ (taking
subscripts modulo $n$). Then the boundary components of $S^{\circ}_\Gamma$
correspond to the faces of $\Gamma$.

An embedding $\iota:\Gamma\rightarrow S$ is said to be a \emph{filling
  embedding} if
$S\setminus |\Gamma|=\bigsqcup_{f\in F} D_f,$ where $|\Gamma|$ denotes the
image of $\iota$, each $D_f$ is a disc, and $F$ is a finite set. That is, the
complement of the embedding is a disjoint union of finitely many discs.

\begin{proposition} (see~\cite[2.2.7]{Labourie}) \label{pro:embeddingexists}
Every ribbon graph has a filling embedding into a compact oriented surface such
that the connected components of $S\setminus |\Gamma|$ are in bijection with the
faces of $\Gamma$ in the above sense.
\end{proposition}

This is proved by gluing discs onto the ribbon surface of the ribbon graph to
fill in the boundary components. Such an embedding has the following uniqueness
property:

\begin{proposition} (see~\cite[2.2.10]{Labourie}) \label{pro:fillinguniqueness}
Let $\Gamma\rightarrow S$, $\Gamma'\rightarrow S'$ be filling embeddings of
ribbon graphs of compact oriented surfaces $S,S'$ and $\varphi:\Gamma\rightarrow
\Gamma'$ an isomorphism of ribbon graphs.  Then $\varphi$ induces an
orientation-preserving homeomorphism $\varphi:|\Gamma|\rightarrow |\Gamma'|$
extending to a homeomorphism from $S$ to $S'$.
\end{proposition}

\begin{corollary} \label{c:fillingembedding}
If $\Gamma$ is a ribbon graph, then there is a compact oriented surface
$S_{\Gamma}$ together with a filling embedding $\Gamma\rightarrow S_{\Gamma}$,
unique up to homeomorphism.
\end{corollary}

Thus we see that there is a filling embedding of an arbitrary Brauer graph
(without, or ignoring, multiplicity) into an oriented surface in such a way
that the cyclic ordering around each vertex arises from the orientation of
the surface. In Section~\ref{sec:cluster}, we shall compare Kauer moves in
this context with a certain kind of twist (or mutation) arising in cluster
theory.

\begin{remark}
Note that~\cite[\S3]{Antipov} has also associated a surface to a Brauer graph
by associating a CW-complex, called the Brauer complex, to the quiver with
relations of the corresponding graph algebra. Comparing the definitions,
it can be seen that this gives rise to the same filling embedding as
above (Corollary~\ref{c:fillingembedding}). In particular, the $G$-cycles
in~\cite[\S2]{Antipov06} correspond to the faces of the Brauer graph as a
ribbon graph.
\end{remark}

Finally, we have the following:

\begin{proposition} (stated as~\cite[2.2.12]{Labourie})
\label{pro:admitsfilling}
Every compact oriented surface admits a filling ribbon graph.
\end{proposition}

One way to realise a filling ribbon graph as in Proposition
\ref{pro:admitsfilling} is through an (ideal) triangulation of the surface.

We now suppose that $(S,M)$ is a compact oriented surface with marked points
$M$ in $S$. We assume in addition that $(S,M)$ is not a sphere with $1$ or
$2$ marked points and that each connected component of $S$ contains at least
one element of $M$ to ensure that $(S,M)$ has at least two triangulations
(see~\cite[\S2]{FominShapiroThurston08}: note that we allow a sphere
with $3$ marked points here).

Let $\T$ be a triangulation of $(S,M)$ where the set of marked points $M$
coincides with the set of vertices of the triangulation.  Then we obtain a
filling ribbon graph $\Gamma_{\T}$ whose vertices are the elements of $M$ and
whose edges are the arcs in $\T$.  Denote by $A_{\T}$ the corresponding
Brauer graph algebra.

We recall the definition of the \emph{flip} of a triangulation
from~\cite[\S2.2]{Burman99}.  Let $a$ be an arc in $\T$ incident with two
distinct triangles $T_1$ and $T_2$ in $\T$ (we call such an arc \emph{flippable}).
Then there is a map $\psi$ from a
square with diagonal $d$ onto the union of $T_1$ and $T_2$ in $S$. The map
$\psi$ may fold the square, possibly identifying distinct points or edges. Then
in the flip of $\T$ at $a$, $a$ is replaced with $\psi(d')=a'$, where $d'$ is
the diagonal of the square distinct from $d$.  See
Figure~\ref{fig:twotriangles}.

\begin{figure}
\psfragscanon \psfrag{W}{$W$} \psfrag{X}{$X$} \psfrag{Y}{$Y$} \psfrag{Z}{$Z$}
\psfrag{D}{$d$} \psfrag{D'}{$d'$} \psfrag{phi}{$\phi$} \psfrag{S}{$S$}
\begin{center}
\includegraphics[width=5cm]{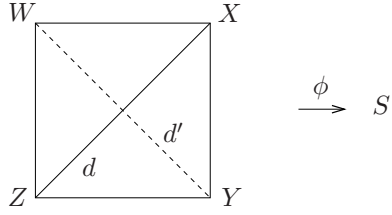}
\end{center}
\caption{The flip of a triangulation.}
\label{fig:twotriangles}
\end{figure}

We recall that, by~\cite[\S2.1]{Burman99}, the possible triangles that can
appear in $\T$ are as in Figure~\ref{fig:possibletriangles}.

\begin{figure}
\psfragscanon \psfrag{A1}{$A_1$} \psfrag{A2}{$A_2$} \psfrag{A3}{$A_3$}
\psfrag{A4}{$A_4$}
\begin{center}
\includegraphics[width=10cm]{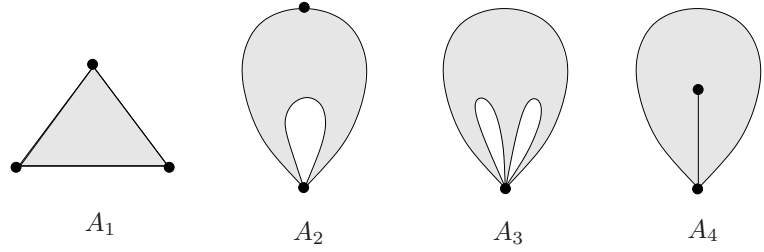}
\end{center}
\caption{The possible triangles in $\T$.}
\label{fig:possibletriangles}
\end{figure}

\begin{lemma} \label{l:successors} With the notation in Figure~\ref{fig:twotriangles}, the following successor relations hold
\begin{enumerate}
\item[(a)] The clockwise successor of $a=\psi(d)$ at $\psi(X)$ is $\psi(WX)$.
\item[(b)] The clockwise successor of $a=\psi(d)$ at $\psi(Z)$ is $\psi(ZY)$.
\end{enumerate}
\end{lemma}

\begin{proof}
This holds because $\psi$ only acts by folding and identifying, so doesn't
change successor relations. We describe this in detail.

Consider the triangle $WXZ$, which is folded by $\psi$ according to $A_1$, $A_2$,
$A_3$ or $A_4$ in Figure~\ref{fig:possibletriangles}. The argument for the
triangle $XYZ$ is similar. We just need to check that $\psi$ preserves the
clockwise successors at each vertex.

\emph{Case $A_1$}: This case is trivial.

\emph{Case $A_2$}: This is illustrated in Figure~\ref{fig:caseA2}.  We indicate
the successor relation by curved arrows, labelled $\alpha, \beta$ and $\gamma$
to indicate the correspondence under $\psi$.  Note that there are three
possibilities (depending on which edge becomes the inside edge in the folding),
but the argument in each case is similar.

\begin{figure}
\psfragscanon \psfrag{W}{$\scriptstyle W$} \psfrag{X}{$\scriptstyle X$}
\psfrag{Z}{$\scriptstyle Z$} \psfrag{psi}{$\scriptstyle \psi$}
\psfrag{psi(WX)}{$\scriptstyle \psi(WX)$} \psfrag{psi(XZ)}{$\scriptstyle
  \psi(XZ)$} \psfrag{psi(ZW)}{$\scriptstyle \psi(ZW)$}
\psfrag{psi(X)}{$\scriptstyle \psi(X)$} \psfrag{psi(W)=psi(Z)}{$\scriptstyle
  \psi(W)=\psi(Z)$} \psfrag{1}{$\scriptstyle \alpha$} \psfrag{2}{$\scriptstyle
  \beta$} \psfrag{3}{$\scriptstyle \gamma$}
\begin{center}
\includegraphics[width=10cm]{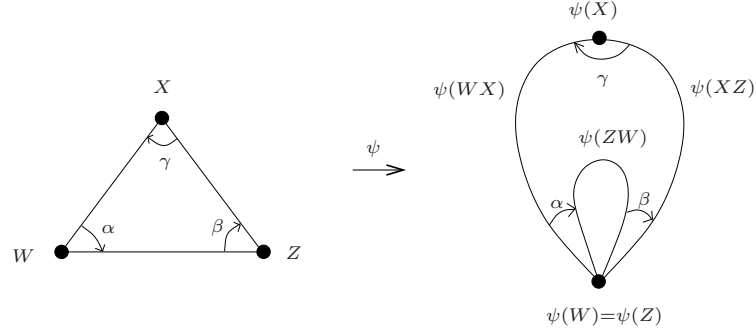}
\end{center}
\caption{Proof of Lemma~\ref{l:successors}, case $A_2$.}
\label{fig:caseA2}
\end{figure}

\emph{Case $A_3$}: This is indicated in Figure~\ref{fig:caseA3}.  We represent
the folding in two steps, first folding as in case $A_2$ (labelled $\psi_0$),
then a further folding, labelled $\psi_1$, to obtain $\psi=\psi_1\circ \psi_0$.

\begin{figure}
\psfragscanon \psfrag{W}{$\scriptstyle W$} \psfrag{X}{$\scriptstyle X$}
\psfrag{Z}{$\scriptstyle Z$} \psfrag{psi0}{$\scriptstyle \psi_0$}
\psfrag{psi1}{$\scriptstyle \psi_1$} \psfrag{psi0(WX)}{$\scriptstyle
  \psi_0(WX)$} \psfrag{psi0(XZ)}{$\scriptstyle \psi_0(XZ)$}
\psfrag{psi0(ZW)}{$\scriptstyle \psi_0(ZW)$} \psfrag{psi0(Z)}{$\scriptstyle
  \psi_0(Z)$} \psfrag{psi0(X)=psi0(W)}{$\scriptstyle \psi_0(X)=\psi_0(W)$}
\psfrag{psi0(X)}{$\scriptstyle \psi_0(X)$}
\psfrag{psi0(W)=psi0(Z)}{$\scriptstyle \psi_0(W)=\psi_0(Z)$}
\psfrag{psi(WX)}{$\scriptstyle \psi(WX)$} \psfrag{psi(XZ)}{$\scriptstyle
  \psi(XZ)$} \psfrag{psi(ZW)}{$\scriptstyle \psi(ZW)$}
\psfrag{psi(X)}{$\scriptstyle \psi(X)$}
\psfrag{psi(W)=psi(X)=psi(Z)}{$\scriptstyle \psi(W)=\psi(X)=\psi(Z)$}
\psfrag{psi(X)=psi(W)=psi(Z)}{$\scriptstyle \psi(X)=\psi(W)=\psi(Z)$}
\psfrag{1}{$\scriptstyle \alpha$} \psfrag{2}{$\scriptstyle \beta$}
\psfrag{3}{$\scriptstyle \gamma$}
\begin{center}
\includegraphics[width=15cm]{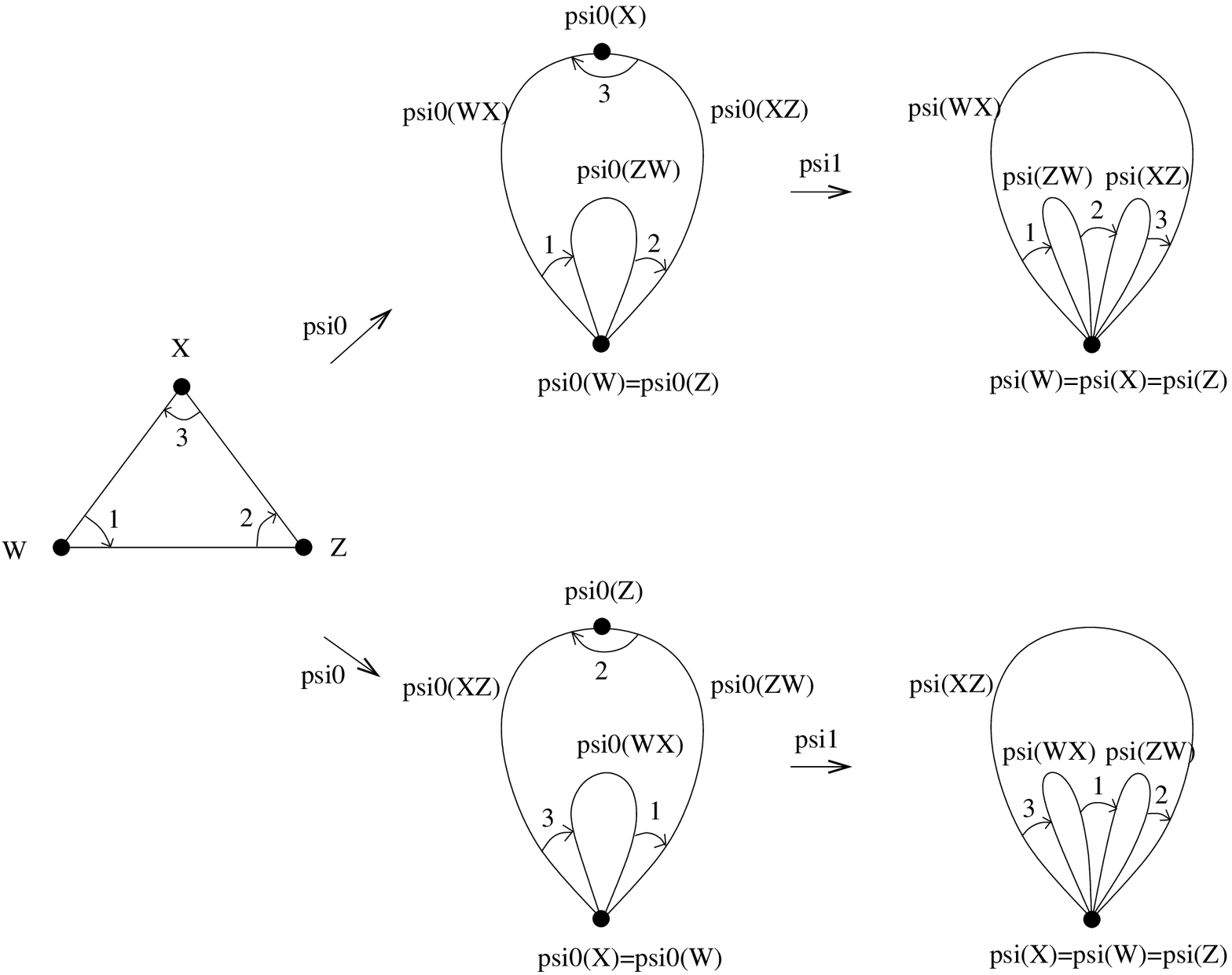}
\end{center}
\caption{Proof of Lemma~\ref{l:successors}, case $A_3$.}
\label{fig:caseA3}
\end{figure}

\emph{Case $A_4$}: This is indicated in Figure~\ref{fig:caseA4}.  Note that, in
this case, $\psi$ must map $XZ$ to the outer edge of the triangle of type $A_4$,
as $\psi(XZ)$ must be a flippable edge in $\T$.

\begin{figure}
\psfragscanon \psfrag{W}{$\scriptstyle W$} \psfrag{X}{$\scriptstyle X$}
\psfrag{Z}{$\scriptstyle Z$} \psfrag{psi}{$\scriptstyle \psi$}
\psfrag{psi(XZ)}{$\scriptstyle \psi(XZ)$} \psfrag{psi(WX)=psi(ZW)}{$\scriptstyle
  \psi(WX)=\psi(ZW)$} \psfrag{psi(X)=psi(Z)}{$\scriptstyle \psi(X)=\psi(Z)$}
\psfrag{psi(Y)}{$\scriptstyle \psi(Y)$} \psfrag{psi(W)}{$\scriptstyle \psi(W)$}
\psfrag{1}{$\scriptstyle \alpha$} \psfrag{2}{$\scriptstyle \beta$}
\psfrag{3}{$\scriptstyle \gamma$}
\begin{center}
\includegraphics[width=10cm]{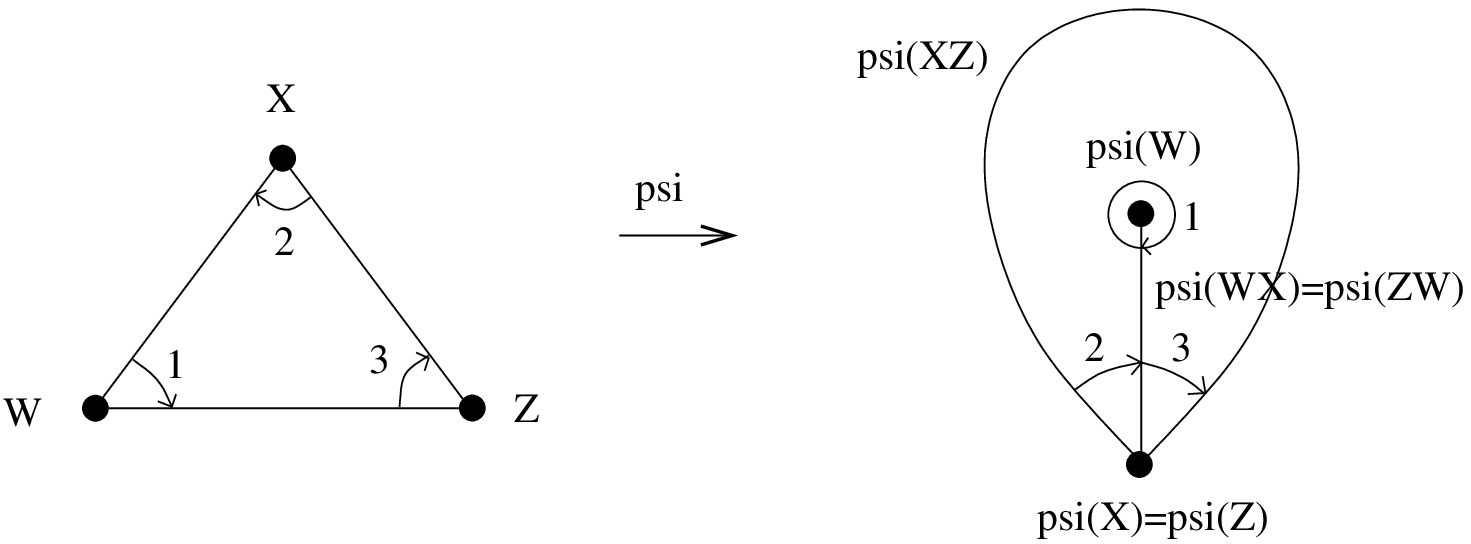}
\end{center}
\caption{Proof of Lemma~\ref{l:successors}, case $A_4$.}
\label{fig:caseA4}
\end{figure}
\end{proof}

\begin{proposition} \label{pro:Kauerflip}
Let $\T$ be a triangulation of a marked compact oriented surface $(S,M)$.
Regarding $\T$ as a Brauer graph, the flip of $\T$ at an arc $a$ coincides with
applying the Kauer move to $\T$ at $a$.
\end{proposition}

\begin{proof}
This follows from Lemma~\ref{l:successors} and the definition of a Kauer move
(see Theorem~\ref{thm:KauerMove}).
\end{proof}

It is a consequence of Proposition~\ref{pro:Kauerflip} that given a marked surface $(S,M)$, up to derived equivalence, there exists a unique Brauer graph algebra associated to $(S,M)$:

\begin{corollary} \label{cor:derivedinvariant}
Let $(S,M)$ be a marked compact oriented surface. Then the Brauer graph algebra
$A_{\T}$ associated to $(S,M)$ via a triangulation $\T$ does not depend on the
choice of $\T$, up to derived equivalence.
\end{corollary}

\begin{proof}
This follows from Proposition~\ref{pro:Kauerflip} and
Theorem~\ref{thm:KauerMove}, using the fact that any two triangulations of
$(S,M)$ are connected by a sequence of flips (see~\cite[Theorem
  2]{Burman99},~\cite[Prop. 3.8]{FominShapiroThurston08}).
\end{proof}


\subsection{Marked surfaces with boundary}
\label{s:boundarysurfaces}
We first note that every oriented surface with boundary is homeomorphic to a
surface with boundary obtained from a surface without boundary by removing a
disjoint collection of discs.

Let $(S,M)$ be a marked surface with (non-empty) boundary.  The marked points
$M$ can be anywhere on the surface, on or off the boundary.  We assume that
there is at least one marked point on each connected component of $S$ and on
each connected component of the boundary of $S$.

Since we would like to consider flips of triangulations of surfaces, we assume,
that $(S,M)$ is not one of the following degenerate cases, in order to ensure
that $(S,M)$ has at least two triangulations
(see~\cite{FominShapiroThurston08}).

\begin{enumerate}[(a)]
\item a sphere with $1$ or $2$ marked points;
\item a monogon with $0$ or $1$ marked points;
\item a digon with no marked points;
\item a triangle with no marked points.
\end{enumerate}

In particular,
\begin{proposition}
Every marked oriented surface with boundary admits a filling ribbon graph.
\end{proposition}

\begin{proof}
As in Proposition~\ref{pro:admitsfilling}, take a triangulation of $(S,M)$ with vertex set $M$, and
include boundary arcs (i.e.\ parts of the boundary of $S$ between adjacent,
possibly equal, marked points on the boundary) in the graph.
\end{proof}

Some of the faces of the corresponding ribbon graph will correspond to boundary
components. So we have a distinguished subset, $F_B$ of the set of faces.

We define a \emph{ribbon graph with boundary} to be a tuple
$\widetilde{\Gamma}=(\Gamma,I,\{s_v\}_{v\in V},F_B)$, where
$(\Gamma,I,\{s_v\}_{v\in V})$ is a ribbon graph and $F_B$ is a subset of
its set of faces with the property that, for every undirected edge
$(e,\overline{e})$ of $\Gamma$, at most one of $e,\overline{e}$ lies in a
face in $F_B$.

Let $S_{\Gamma,I,\{s_v\}}=S_{\Gamma}$ be the surface constructed as in
the ribbon graph case, in which there is a filling embedding of $\Gamma$ (see
Proposition~\ref{pro:embeddingexists}).  Then $S_{\Gamma}\setminus |\Gamma|=
\sqcup_{f\in F}D_f$ where $F$ is a finite set and each $D_f$ is an open disc.
We define $S_{\widetilde{\Gamma}}$ to be $S_{\Gamma}\setminus \cup_{f\in
  F_B}D_f$, a surface with boundary obtained by removing the discs from
$S_{\Gamma}$ corresponding to the faces in $F_B$.
Then $\widetilde{\Gamma}\rightarrow S_{\widetilde{\Gamma}}$ is a
filling embedding of $\widetilde{\Gamma}$ into $S_{\widetilde{\Gamma}}$.

We extend the notion of morphism of ribbon graphs to ribbon graphs with boundary
in the natural way (i.e.\ such maps should preserve the marked set, $F_B$, of
faces). A morphism of ribbon graphs with boundary is an isomorphism if it has an
inverse which is also a morphism of ribbon graphs with boundary.

\begin{lemma} \label{l:fillingboundaryuniqueness}
Let $\widetilde{\Gamma}$ and $\widetilde{\Gamma'}$ be ribbon graphs with
boundary together with filling embeddings $\widetilde{\Gamma}\rightarrow
\widetilde{S}$ and $\widetilde{\Gamma'}\rightarrow \widetilde{S'}$ mapping the
boundary faces to the boundaries of the boundary components in each case.  If
$\varphi:\widetilde{\Gamma}\rightarrow \widetilde{\Gamma'}$ is an isomorphism of
ribbon graphs with boundary then it induces a homeomorphism
$\varphi:|\widetilde{\Gamma}|\rightarrow |\widetilde{\Gamma'}|$ which extends to
a homeomorphism from $\widetilde{S}$ to $\widetilde{S'}$.
\end{lemma}

\begin{proof}
Let $S,S'$ be the surfaces without boundary obtained from $\widetilde{S}$ and
$\widetilde{S'}$ by gluing discs into their boundary components. The filling
embeddings $\Gamma\rightarrow \widetilde{S}$ and $\Gamma'\rightarrow
\widetilde{S'}$ induce filling embeddings $\Gamma\rightarrow \widetilde{S}$ and
$\Gamma\rightarrow S'$ of the underlying ribbon graphs. By
Proposition~\ref{pro:fillinguniqueness}, there is a homeomorphism
$\varphi:S\rightarrow S'$ restricting to a homeomorphism from $|\Gamma|$ to
$|\Gamma'|$. By the construction of this homeomorphism, it restricts to a
homeomorphism from $\widetilde{S}$ to $\widetilde{S'}$ with the required
properties.
\end{proof}

\begin{proposition}
Let $\widetilde{\Gamma}$ be a ribbon graph with boundary. Then there is an
oriented surface with boundary $S$ and filling embedding
$\widetilde{\Gamma}\rightarrow S$,
mapping the boundary faces to the boundaries of the boundary components,
which is unique up to homeomorphism.
\end{proposition}

\begin{proof}
Existence is guaranteed by the argument preceding
Lemma~\ref{l:fillingboundaryuniqueness}, and uniqueness follows from
Lemma~\ref{l:fillingboundaryuniqueness}.
\end{proof}

We say that an arc in a triangulation of a marked surface with boundary is
\emph{flippable} if it is incident with two triangles
in $\T$ and if it is not a boundary arc.

\begin{corollary}
Let $\T$ be a triangulation of a marked oriented surface with boundary $(S,M)$.
Regarding $\T$ as a Brauer graph, the flip of $\T$ at a flippable arc $a$ in
$\T$ coincides with applying the Kauer move to $\T$ at $a$.
\end{corollary}

\begin{proof}
The proof is the same as the proof of Proposition~\ref{pro:Kauerflip}.
\end{proof}

Similarly to the case of a marked surface without boundary, for a marked surface with boundary, up to derived equivalence, there exists a unique Brauer graph algebra associated to that surface:

\begin{corollary} \label{cor:surfaceinvariant}
Let $(S,M)$ be a marked oriented surface with boundary. Then the derived equivalence
class of the Brauer graph algebra $A_{\T}$ associated to a triangulation $\T$ of
$(S,M)$ does not depend on the choice of triangulation $\T$.
\end{corollary}

\begin{proof}
As for the proof of Corollary~\ref{cor:derivedinvariant}.
\end{proof}

\section{Cluster Algebras}
\label{sec:cluster}

In this section we discuss the relationship between cluster theory and the
results above.

Let $P_n$ be a polygon with $n$ marked points on its boundary, and $\T$ a
triangulation of $P_n$. Then regarding $\T$ as a Brauer graph, there is
a corresponding quiver $Q_{\T}$. For an example in the case $n=7$,
see Figure~\ref{fig:DiscQuiver}. We call arrows outside $P_n$
\emph{boundary arrows}.

\begin{figure}
\includegraphics[width=6cm]{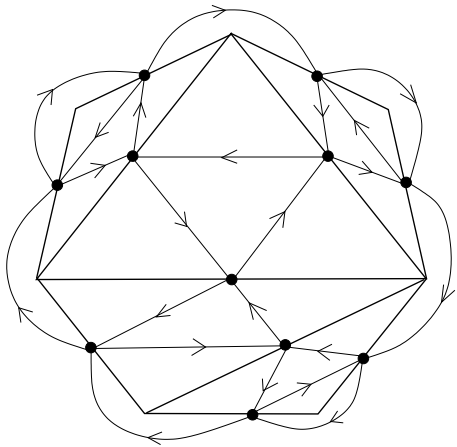}
\caption{Example of the quiver $Q_{\T}$ of a triangulation $\T$.}
\label{fig:DiscQuiver}
\end{figure}

Let $R$ denote the polynomial ring $K[x]$, and let $\Lambda$ denote the
Gorenstein tiled $R$-order considered in~\cite{demonetluo}, defined as a
matrix algebra by the $n\times n$ matrix:
$$
\begin{pmatrix}
R & R & R & \cdots & R & (x^{-1}) \\
(x) & R & R & \cdots & R & R \\
(x^2) & (x) & R & \cdots & R & R \\
\vdots &&& \ddots & & \vdots \\
(x^2) & (x^2) & (x^2) & \cdots & R & R \\
(x^2) & (x^2) & (x^2) & \cdots & (x) & R
\end{pmatrix}.
$$

Let $\F$ be the category of (maximal) Cohen-Macaulay modules over $\Lambda$.
We recall some of the results of~\cite{demonetluo}
(see also~\cite{BaurKingMarsh,JensenKingSu}).

\begin{theorem} \cite[Thm.\ 1]{demonetluo} \label{t:demonetluo}
\begin{enumerate}
\item[(a)]
The category $\F$ is Frobenius and its stable category $\C$ is triangle
equivalent to the cluster category of type $A_{n-3}$.
\item[(b)]
The indecomposable objects in $\F$ are in bijection with
the arcs joining vertices in $P_n$ (including boundary arcs).
\item[(c)]
The bijection in (b) induces a bijection $\T\mapsto M_{\T}$ between the
triangulations in $P_n$ and the cluster-tilting objects in $\F$.
\item[(d)]
For each triangulation $\T$ of $P_n$, the algebra $\End_{\F}(M_{\T})^{\op}$
is a frozen Jacobian algebra associated to a frozen quiver with potential.
\end{enumerate}
\end{theorem}

\begin{remark} \label{r:quiverscoincide}
The quiver of the frozen Jacobian algebra in Theorem~\ref{t:demonetluo}
can be obtained from the quiver $Q_{\T}$ by deleting boundary arrows
whose end points correspond to two sides of the same triangle
of $\T$. The boundary vertices are frozen.
Note also that $Q_{\T}$ coincides exactly with the quiver of the Postnikov diagram
associated to $\T$ in~\cite[\S12]{BaurKingMarsh} using a modified version of
the construction in~\cite[Cor. 2]{Scott06}. The associated dimer algebra
in~\cite{BaurKingMarsh} is isomorphic to the frozen Jacobian algebra referred to
above by~\cite[Lemma 11.1]{BaurKingMarsh}.
\end{remark}

The following is a special case of~\cite[Prop.\ 4]{Palu09}, using
$\Mod \Lambda$ to denote the category of all modules over $\Lambda$.

\begin{theorem} \cite{Palu09} \label{thm:clusterequivalence}
Let $\C$ be a Hom-finite $2$-Calabi-Yau triangulated category which is the
stable category of a Frobenius category $\F$. Let $T,T'$ be cluster-tilting
objects in $\C$ with preimages $M,M'$ in $\F$. Then
$D(\Mod \End(M)^{\op})\simeq D(\Mod \End(M')^{\op})$.
\end{theorem}

Theorem~\ref{thm:clusterequivalence} applies in this case, and we have:

\begin{proposition} \label{pro:brauercluster}
Let $\T,\T'$ be triangulations of $P_n$, and $M_{\T}$, $M_{\T'}$ the
corresponding cluster-tilting objects in $\F$.
Then the categories $D(\Mod \End(M_{\T})^{\op})$ and $D(\Mod \End(M_{\T'})^{\op})$ are
equivalent.
\end{proposition}

\begin{remark} \label{rem:twoequivalences}
It is interesting to compare the above with
Corollary~\ref{cor:surfaceinvariant} (in the disc case).
We see that, in terms of derived equivalences, the algebras $\End(M_{\T})$,
for $\T$ a triangulation of a disc, behave in a similar way to the
corresponding Brauer graph algebras $A_{\T}$.
\end{remark}

Using~\cite[Lemma 3.1]{CalderoChapotonSchiffler06} and
Remark~\ref{r:quiverscoincide}, we have:

\begin{proposition} \label{pro:triangulationflip}
Let $\T$ be a triangulation of $(S,M)$ and $\T'$ the triangulation obtained from
$\T$ by flipping an internal arc, $\gamma$. Then the quiver of the
Brauer graph algebra associated to $\T'$, i.e., $Q_{\T'}$, can be
obtained from $Q_{\T}$ by applying Fomin-Zelevinsky quiver
mutation~\cite{FominZelevinsky02} to $Q_{\T}$ at the vertex corresponding to
$\gamma$, leaving the boundary arrows unchanged.
\end{proposition}

A version of Proposition~\ref{pro:triangulationflip}
holds in a more general context. Suppose that $(S,M)$ is a marked oriented
surface in which every element of $M$ lies on the boundary of $S$
(i.e.\ the \emph{unpunctured} case). Let
$\iota:\Gamma\rightarrow S$ be a filling embedding of a ribbon graph $\Gamma$
into $S$, such that $M=\iota(\Gamma_0)$ and boundary faces are mapped to the
boundaries of the boundary components.

\begin{proposition}
Let $e$ be an edge of $\Gamma$ with image $\gamma=\iota(e)$ not on the
boundary. Then, applying a Kauer move to $\Gamma$ at $e$ corresponds to
twisting $\gamma$ in the sense of~\cite[\S3]{MarshPalu} with respect to
the set of remaining edges of $|\Gamma|$ which are not on the boundary.
If $\Gamma$ is a triangulation, then the change in the quiver $Q_{\Gamma}$
of the Brauer graph algebra of $\Gamma$ under this twist coincides with
Fomin-Zelevinsky quiver mutation~\cite{FominZelevinsky02}.
\end{proposition}

\begin{proof}
The first statement follows from a comparison of the definition of the
Kauer move and the definition of twisting in~\cite[\S3]{MarshPalu}. The
second statement can be checked directly as
in~\cite[Prop.\ 4.8]{FominShapiroThurston08}.
\end{proof}

Note that the twist of arcs considered in~\cite{MarshPalu} is shown to
correspond to a categorical mutation in the cluster
category~\cite{BruestleZhang11} associated to the surface.

We note that A. Dugas~\cite{Dugas} has remarked on similarities between quiver
mutation (in fact mutation of quivers with potential in the sense
of~\cite{DerksenWeymanZelevinsky}) and tilting mutations considered by
Aihara~\cite{Aihara}. Note that these tilting mutations can be
regarded, in the case of the complement of a single vertex, as a special
case of the Kauer derived equivalences~\cite{Kauer98} discussed above.


\section{The Brauer tree of an $m$-angulation of a disc}
\label{sec:m_angulations}

In this section we restrict the surface we consider to the disc with $n$ marked points on the boundary.
Here we allow
$m$-angulations (or even angulations, in general) and not only triangulations of
the disc. Given an $m$-angulation of the disc we define a dual graph of the $m$-angulation. We then
give an explicit tilting complex of the
Brauer graph algebra of this dual graph realizing the change of graph induced by
a Kauer move or mutation of one of the diagonals of the $m$-angulation of the disc.

Let $P_n$ be a polygon with $n$ vertices. For $n =k(m-2)+2$ for some $k,m \in
\mathbb{Z}^+$, an \emph{$m$-angulation} $\M$ of $P_n$ is a collection of arcs
joining the vertices of $P_n$ and dividing it up into $m$-gons. The set of edges
of $P_n$ constitutes the set of \emph{boundary edges} of $\M$. We call an
\emph{internal arc} of $\M$ any arc that is not a boundary edge.

Regarding $\M$ as a (locally) embedded graph in the plane, we have seen that it can
be regarded as a Brauer graph by choosing an orientation of the plane which then
induces an ordering of the edges around each vertex.  As discussed in the
previous sections, unless otherwise stated we choose the clockwise orientation.

With this set-up there is a corresponding Brauer graph algebra which we denote
by $A_{\M}$. Note that $A_{\M}$ has multiplicity function $m \equiv 1$.

Given an $m$-angulation $\M$ and an internal arc $a$, there is a new
$m$-angulation $\mu_a(\M) = \M'$ obtained in the following way.  Removing $a$
leaves a $2m-2$-gon $H$ which had $a$ as one of its diagonals. Then $\M'$ is
obtained by rotating $a$ one-step clockwise within $H$.  This move is a
generalization of the the flip in previous section and it is generally known in
cluster theory as a \emph{mutation} at $a$ (in the $(m-2)$-cluster sense;~\cite{HughThomas07} (see also~\cite[\S11]{BuanThomas09},~\cite{FominReading05}).

The following is easy to check.

\begin{lemma}\label{Kauermutation}
Let $\M$ be an $m$-angulation of $P_n$ and $a$ an internal arc of $\M$. Then
mutating $\M$ at $a$ (in the way described above) coincides with applying the
Kauer move to $\M$ at $a$.
\end{lemma}

We also observe the following:

\begin{lemma}\label{l:mcasequiver}
Let $\M$ be an $m$-angulation of $P_n$ and $a$ an internal arc of $\M$.  Then,
regarding $\M$ as a Brauer graph (with an appropriate orientation of $P_n$), the
full subquiver of $A_{\M}$ on the vertices corresponding to non-boundary arcs
coincides with the quiver of the corresponding $(m-2)$-cluster-tilted algebra as
given by~\cite[Prop.\ 4.1.8]{Murphy}.
\end{lemma}

\begin{corollary}
Let $\M_1$ and $\M_2$ be two $m$-angulations of $P_n$. Then $A_{\M_1}$ and
$A_{\M_2}$ are derived equivalent.
\end{corollary}

\begin{proof}
This follows from Lemma~\ref{Kauermutation} and the fact that any two
$m$-angulations are connected by a sequence of mutations (this can be seen
using~\cite[Prop.\ 7.1]{BuanThomas09} or~\cite[Prop.\ 4.5]{ZhouZhu09}, combined
with~\cite[\S11]{BuanThomas09}).
\end{proof}

We also consider the completed $m$-ary tree, $\MM$, corresponding to $\M$.  Its
interior vertices correspond to the $m$-gons in $\M$, with a leaf vertex
corresponding to each boundary edge. Two interior vertices are connected by an
edge if the corresponding $m$-gons share a common edge in $\M$, and an interior
vertex and a leaf vertex are connected by an edge if the boundary edge
corresponding to the leaf vertex forms part of the boundary of the $m$-gon
corresponding to the interior vertex.

\begin{remark} If we identify all of the leaf vertices in $\MM$,
we obtain the dual graph of $\M$.
\end{remark}

As for $\M$, the graph $\MM$, via its embedding in the plane, can be considered
as a Brauer graph.  So we again have a corresponding Brauer graph algebra,
$A_{\MM}$. Note that $\MM$ is in fact a tree and thus $A_{\MM}$ is a
Brauer tree algebra.

In a similar way, we may also consider $\MM$ without its leaf vertices (deleting
all incident edges also).  We call this non-boundary version $\M^*$.

\begin{remark}
If we adopt the rule for $\M$ that the successor of an edge is anticlockwise
from the edge, while in $\MM$ we use the usual (clockwise) rule then there is a
close correspondence between the quivers of $A_{\M}$ and $A_{\MM}$: the arrows
are the same except that $Q_{\M}$ has an extra cycle of arrows around its
boundary. If, however, we were to consider instead the dual graph of $\M$,
i.e.\ if all of the boundary vertices are identified, then the arrows in the quiver
corresponding to the dual graph would be the same as in $Q_{\M}$.  It will
become clear in Section~\ref{sec:counterexamples} why we choose to work with $\MM$
instead of the dual graph of $\M$.
\end{remark}

Given an internal edge $a$ of $\M$, denote by $a^*$ the unique edge of $\MM$
intersecting $a$.

\begin{definition}
Mutating $\M$ at an internal edge $a$, induces a corresponding move on $\MM$:
the first edge clockwise of $a^*$ at each end-point of $a^*$ is moved along
$a^*$ (together with the subtrees attached to the edges we are moving) so that
it becomes incident with the other end-point of $a^*$ instead.  We call this a
\emph{dual Kauer move} at $a^*$ and we denote the resulting tree by
$\mu^*_{a^*}(\MM)$.
\end{definition}

\begin{figure}[!h]
\psfragscanon \psfrag{es}{\begin{footnotesize}$a^*$\end{footnotesize}}
\begin{center}
\includegraphics[width=7cm]{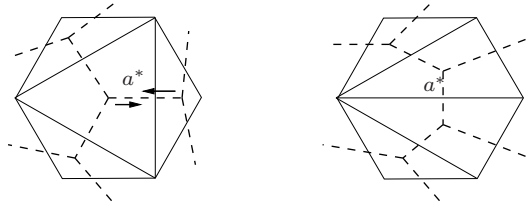} 
\end{center}
\caption{Example of a dual Kauer move at the edge $a^*$. The dotted lines in the
  left figure denote $\MM$ and the dotted lines in the right figure denote
  $\mu^*_{a^*}(\MM)$. }
\label{DualKauerMove}
\end{figure}

\begin{remark}
\begin{enumerate}
\item[(1)]
The dual Kauer move is known in graph theory under the name of
nearest neighbour interchange (NNI)~\cite[\S2]{watermansmith73}
or also as a Whitehead move.
\item[(2)]
Note that this rule does not apply to $\M^*$. For an example of this, see
Figure~\ref{NotDualKauer} where going from the left hand figure to the right
hand figure, the change of in $\M^*$ induced by a Kauer move in the
heptagon at the edge corresponding to $a^*$ is not a dual Kauer move.
But notice that going the other way, that is from the right hand
figure to the left hand one, if we mutate the edge corresponding to $a^*$ in $\M$
and then take the dual graph this does correspond to a dual Kauer move.
\end{enumerate}
\end{remark}

\begin{figure}[!h]
\psfragscanon \psfrag{es}{\begin{footnotesize}$a^*$\end{footnotesize}}
\begin{center}
\includegraphics[width=6cm]{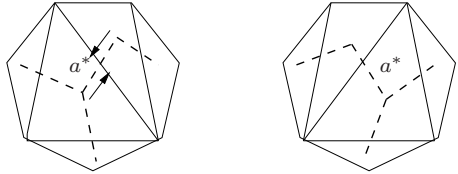} 
\end{center}
\caption{Mutating at $a$ in $\M$ does not induce a dual Kauer move on $\M^*$
(the latter indicated by the arrows in the left hand picture).}
\label{NotDualKauer}
\end{figure}

In the following we see that when working in the dual context of $\MM$, a dual Kauer move
also induces a derived equivalence.
Note that by~\cite{Rickard89} it is already known that the
two corresponding Brauer tree algebras are derived equivalent, since they both have the same
number of vertices and the same multiplicity function.  However, our interest here is in
constructing an explicit derived equivalence compatible with the geometry.

In general, for a finite dimensional algebra $A$, let $\P (A)$ denote the
category of finitely generated projective $A$-modules and let $\K^{\flat}(\P(A))$
denote the bounded homotopy category of complexes of projective $A$-modules.
Recall that an object $T$ in $\K^{\flat}(\P(A))$ is said to be a
\emph{tilting complex} provided:
\begin{enumerate}[(i)]
\item $\Hom _{\K^{\flat}(\P(A))} (T, T[n]) = 0$ for all $n \neq 0$, and
\item $T$ generates $\K^{\flat}(\P(A))$ as a triangulated category.
\end{enumerate}

Following Rickard~\cite{Rickard89Morita}, in order to show that two
two finite dimensional algebras $A$ and $A'$ are derived equivalent it is
enough to show that there is a tilting complex $T$ in $\K^{\flat}(\P(A))$
such that $A'$ is isomorphic to $\End _{K^{\flat}(\P(A))} (T) ^{\rm opp}$.

Consider a dual Kauer move at an edge $a$ in $\MM$.
Let $X$ and $Y$ denote the two vertices in $\MM$ incident with $a$.
We label the edges near $a$ as follows.
In clockwise orientation around $X$ we have the following edges:
$a, b, d_1, \dots, d_{m-2}$ after which we return to $a$. Similarly,
around $Y$, we have $a, c, e_1, \dots, e_{m-2}$, after which we return to $a$.
We call this \emph{Configuration I}, see figure~\ref{fig: Configuration 1}.

\begin{figure}[h!]
\psfragscanon \psfrag{a}{$\scriptstyle a$} \psfrag{b}{$\scriptstyle b$}
\psfrag{c}{$\scriptstyle c$} \psfrag{d1}{$\scriptstyle d_1$}
\psfrag{e1}{$\scriptstyle e_1$} \psfrag{dm-2}{$\scriptstyle d_{m-2}$}
\psfrag{em-2}{$\scriptstyle e_{m-2}$} \psfrag{h}{$\scriptstyle h$}
\psfrag{i}{$\scriptstyle i$} \psfrag{T}{$ \T^\vee$} \psfrag{U}{$
  \mu^*_{a^*}(\T^\vee)$}
\begin{center}
\includegraphics[width=5cm]{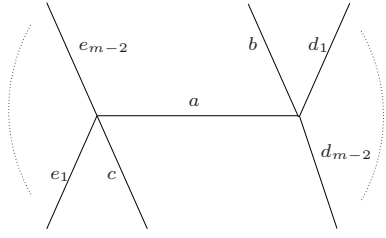}
\end{center}
\caption{Configuration I: Local configuration around the edge $a$.}
\label{fig: Configuration 1}
\end{figure}

 After a dual Kauer move at $a$
(recall that this corresponds to mutating the only edge of the $m$-angulation
of the polygon that
the edge $a$ intersects and then taking our version of the dual graph),
we have the following configuration, which we call
\emph{Configuration II}: around $X$ we have $a, d_1, \dots, d_{m-2}, c$
and then we return to $a$ and around $Y$ we have $a, e_1, \dots, e_{m-2}, b$
and then we return to $a$. See also figure~\ref{fig: Configuration 2}.

\begin{figure}[h!]
\psfragscanon \psfrag{a}{$\scriptstyle a$} \psfrag{b}{$\scriptstyle b$}
\psfrag{c}{$\scriptstyle c$} \psfrag{d1}{$\scriptstyle d_1$}
\psfrag{e1}{$\scriptstyle e_1$} \psfrag{dm-2}{$\scriptstyle d_{m-2}$}
\psfrag{em-2}{$\scriptstyle e_{m-2}$} \psfrag{h}{$\scriptstyle h$}
\psfrag{i}{$\scriptstyle i$} \psfrag{T}{$ \T^\vee$} \psfrag{U}{$
  \mu^*_{a^*}(\T^\vee)$}
\begin{center}
\includegraphics[width=5cm]{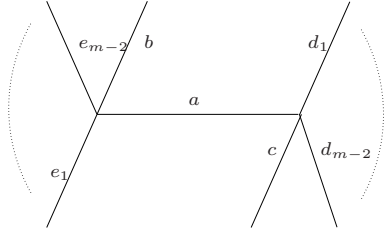}
\end{center}
\caption{Configuration II: Local configuration around the edge $a$ after application of dual Kauer move at edge $a$.}
\label{fig: Configuration 2}
\end{figure}

Let $I$ be the set of edges of $\MM$.
For $i=b, d_1, \ldots, d_{m-2}, c, e_1, \ldots, e_{m-2}$, let $G_i$ be the
subtree of $\MM$ at the vertex of the edge $i$ which is not incident with
$a$, together with $i$ itself.

Define an object $T= \bigoplus_{i \in I} T_i$ in $\K^{\flat}(\P(A_{\MM}))$ as
follows:

$$\begin{array}{lll} T_a & := & 0 \rightarrow P_a; \\ T_b & := & P_b \rightarrow
  P_a; \\ T_c & := & P_c \rightarrow P_a; \\ T_i & := & \left\{ \begin{array}{ll}
    P_i \rightarrow 0, & \mbox{if $i \in G_b \setminus \{b\}$ or $i \in G_c
      \setminus \{c\}$; }\\ 0 \rightarrow P_i, & \mbox{if $i \in G_{d_j}$ or $i \in
      G_{e_j}$ for some $j$,} \\
\end{array} \right. \\
\end{array}
$$
where $P_s$ denotes the projective indecomposable module at vertex $v_s$ in the
quiver of $A_{\MM}$ (which corresponds to the edge $s$ in $\MM$).
For an object $X$ in an additive category, we denote by $\add(X)$ the
full subcategory consisting of direct summands of finite direct sums of copies
of $X$.
\begin{remark} \label{rem:inducedequivalence}
\begin{enumerate}[(a)]
\item Let $A'=End_{K^{\flat}(\P(A_{\MM}))}(T)$.
For $i\in I$, let $P'_i=\Hom_{K^{\flat}(\P(A_{\MM}))}(T,T_i)$ be an indecomposable
projective $A'$-module and let $P' = \bigoplus _{i \in I} P'_i$.
Then the functor $\Hom_{K^{\flat}(\P(A_{\MM}))}(T, -)$ induces an equivalence
of categories $\add(T) \rightarrow \add(P')$.
\item
We note that according to our conventions the arrows in the quiver of $A_{\MM}$
correspond to the clockwise orientation in $\MM$, so that homomorphisms
between the projective indecomposables go anticlockwise to the orientation in
$\MM$.
\end{enumerate}
\end{remark}

The fact that $T = \bigoplus _{i \in I} T_i$ in $K^{\flat}((P(A_{\MM}))$ is a
tilting complex is easily verified and also follows directly from the fact that $T$ is an Okuyama-Rickard complex
\cite{Okuyama}.
 It then follows
from~\cite{Rickard89Morita} that $A_{\MM}$ and $A' = \End _{K^{\flat}((P(A_{\MM}))} (T)^{\rm opp}$ are derived equivalent.

We now show that $A'$ is the Brauer tree algebra obtained from $A$ through a dual Kauer move at the edge $a$.

\begin{lemma}\label{Brauertree}
$A'$ is a Brauer tree algebra with no exceptional vertex.
\end{lemma}

\begin{proof}
Firstly, note that $A$ is symmetric, as it is a Brauer tree
algebra. Rickard~\cite{Rickard89} states that if a symmetric finite dimensional
algebra $\Lambda$ is derived equivalent to another finite dimensional algebra
$\Lambda'$, then $\Lambda'$ is also symmetric. Hence, in our case, $A'$ is also
symmetric.
Next,~\cite[4.2]{Rickard89} states that a Brauer tree algebra is determined up
to derived equivalence by its exceptional vertex multiplicity and its number of
edges.  Hence $A$ is derived equivalent to the Brauer tree algebra $A''$
corresponding to a star with multiplicity 1. Thus $A'$ is derived equivalent to
$A''$. By~\cite[2.2]{Rickard89} if two finite dimensional symmetric algebras
$A'$ and $B'$ are derived equivalent then they are stably equivalent. Hence $A'$
and $A''$ are stably equivalent.  According to \cite[X,
  3.14]{AuslanderReitenSmalo} if two finite dimensional algebras $\Lambda$ and
$\Lambda'$ are stably equivalent and if $\Lambda$ is a symmetric Nakayama Brauer
tree algebra (\emph{e.g.} a star) then $\Lambda'$ is a Brauer tree
algebra. Hence $A'$ is a Brauer tree algebra and by~\cite[4.2]{Rickard89} it has
multiplicity function equal to 1.
\end{proof}

\begin{theorem}\label{thm:flipderivedequivalent}
There is an isomorphism of algebras from $A' = \End_{K^{\flat}(P(A_{\MM}))}(T)^{\rm opp}$ to
$A_{\overline{\mathcal{N}}^*}$, where $\overline{\mathcal{N}}^* = \mu^*_a(\MM)$, that is
$\overline{\mathcal{N}}^*$ is obtained from $\MM$ by a dual Kauer move on the edge $a$.
Hence, $A_{\overline{\mathcal{M}}^*}$ and $A_{\overline{\mathcal{N}}^*}$ are
derived equivalent.
\end{theorem}

\begin{lemma}\label{edges}
With the notation and the set-up of Theorem~\ref{thm:flipderivedequivalent}, suppose
for a summand $T_i$ of $T$ we find two non-zero loops of morphisms
\begin{enumerate}[(a)]
\item $T_i \rightarrow T_{i_1} \rightarrow T_{i_2} \rightarrow \dots \rightarrow
T_{i_r} \rightarrow T_i$;
\item $T_i \rightarrow T_{j_1} \rightarrow T_{j_2} \rightarrow \dots \rightarrow
T_{j_s} \rightarrow T_i$
\end{enumerate}
and show that \\
(c) $\Hom _{K^{\flat}(P(A_{\MM}))} (T_i, T_k) = 0$ for all $k$ such that $T_k$
is not one of the components in the loops (a) and (b) above.

Then, the edges incident with the end-points of $i$ in the Brauer graph of $A'$
must be, in the clockwise orientation, $i, i_1, \dots, i_r,i_1,i$ on one vertex of
$i$ and $i,j_1, \dots, j_s,i$ on the other vertex of $i$.
\end{lemma}

\begin{proof}
By Remark~\ref{rem:inducedequivalence}, the edges in the Brauer graph
of $A'$ can be identified with those in the Brauer graph of $A$.
By (a),(b) and (c), the edges incident
with $i$ must be $\{i_1,\ldots ,i_r,j_1,\ldots ,j_s\}$
(since $\Hom_{A_{\MM}} (T_i, T_k) = 0$ implies that $\Hom_{A'}(P'_i, P'_k)=0$; see Remark~\ref{rem:inducedequivalence}(a)).

Since the Brauer graph of $A'$ is a tree with multiplicity function equal to 1
there are no loops and the edges around a vertex appear only once, i.e.\
there are no repeats.
Since $\Hom(T_{i_1},T_{i_2})\not=0$, $i_2$ must be incident with the same end of $i$ as $i_1$ by Lemma~\ref{Brauertree}.
If the edge $i_2$ was clockwise of $i$ and anticlockwise of $i_1$ (about the common vertex), then the map $T_i \rightarrow T_{i_1}$ would factor through $T_{i_2}$.
Therefore the map in (a) would factor $T_i \rightarrow T_{i_2} \rightarrow
T_{i_1} \rightarrow T_{i_2} \rightarrow T_i$.  However, this composition is
zero by Remark~\ref{rem:inducedequivalence} and Lemma~\ref{Brauertree}
as any composition that is longer than a cycle is zero in a Brauer tree
algebra with multiplicity function 1. Thus we get a contradiction, and 
hence $i_2$ lies clockwise of $i_1$ and anticlockwise of $i$.
If we repeat this argument for all the edges incident
with $i$, the claim follows.
In the case of a leaf, there is only one
non-zero loop, so (a) and (c) imply the clockwise ordering of the edges
around the vertex of $i$ which is not
of valency $1$.
\end{proof}

\emph{Proof of Theorem~\ref{thm:flipderivedequivalent}.}  By
Lemma~\ref{Brauertree}, $A'$ is a Brauer tree algebra with multiplicity function equal to 1.
We now apply the technique of Lemma~\ref{edges} to each indecomposable summand
$T_i$ of $T$. We consider the following compositions in $K^{\flat}(A_{\MM})$
$$ \xymatrix{T_a \ar[r] & T_b \ar[r] & T_{e_{m-2}} \ar[r] & T_{e_{m-3}} \ar[r] &
  \cdots \ar[r] & T_{e_{1}} \ar[r] & T_a}$$
and
$$\xymatrix{T_a \ar[r] & T_c \ar[r] & T_{d_{m-2}} \ar[r] & T_{d_{m-3}} \ar[r] &
 \cdots \ar[r] & T_{d_{1}} \ar[r] & T_a }$$
It is easy to see that we get a non-zero
composition in both cases. Now we check that condition (c) of
the claim above holds. If $j \in G_b \setminus \{ b\} \cup G_c \setminus
\{ c\}$ then any map of complexes from
$T_a=(0 \rightarrow P_a)$ to $T_j=(P_j \rightarrow 0)$ must be zero.
If $j \in G_{d_k} \setminus \{ d_k\} \cup G_{e_k} \setminus \{ e_k\}$
for some $k$, then any map of complexes from
$T_a=(0 \rightarrow P_a)$ to $T_j=(0 \rightarrow P_j)$ must be zero
since, by the defining relations of $A_{\MM}$, there is no non-zero map
$P_a \rightarrow P_j$.
Thus, in both of these cases we have
$\Hom_{K^{\flat}((P(A_{\MM}))} (T_a, T_j) = 0$.
It follows that the edges around $a$ are as in Configuration II above.

Next we consider non-zero loops at $T_b$.
The edge $b$ has two vertices $X$ and $B$. Recall that the edges around vertex
$X$ are given in clockwise order by $b, d_1, d_2, \ldots, d_{m-2}, a$ and back
to $b$. Let the edges around vertex $B$ be given by $b, x_1, x_2, \ldots,
x_{m-1}$ and back to $b$. After the dual Kauer move at the edge $a$, we
have the following configuration of edges around the vertices of $b$. At vertex
$X$ we have in clockwise order $b, a, e_1, e_2, \ldots, e_{m-1}$ and back to $b$
and the order around $B$ remains unchanged. It is then easy to see that we have the
 following two non-zero loops of morphisms starting and ending
at $T_b$
$$\xymatrix{T_b \ar[r] & T_{e_{m-2}} \ar[r] & T_{e_{m-3}} \ar[r] & \cdots
  \ar[r] & T_{e_{1}} \ar[r] & T_a \ar[r] & T_{b}}$$
and
$$\xymatrix{T_b \ar[r] & T_{x_{m-1}} \ar[r] & T_{x_{m-2}} \ar[r] & \cdots \ar[r]
  & T_{x_{1}} \ar[r] & T_b.} $$
If $j \in G_b \backslash \{ b, x_1, \ldots, x_{m-1} \} \cup G_c$, then there
is no non-zero map $T_b \rightarrow T_j$ since \sloppy $\Hom_{A_{\MM}}(P_b, P_j)
= 0$. Similarly, if $j \in G_{d_k} \backslash \{ d_k \}$ or $G_{e_k}$ for some
$k$, then there is no non-zero map $T_b \rightarrow T_j$ in
$K^{\flat}(\P(A_{\MM*}))$ since $\Hom_{A_{\MM}}(P_a, P_j) = 0$.
If $j = d_k$, then there is no non-zero map $T_b \rightarrow T_j$ in
$K^{\flat}(\P(A_{\MM}))$ since the corresponding diagram
$$\xymatrix{P_b \ar[d] \ar[r] & P_a \ar[d] \\ 0 \ar[r] & P_j}$$ cannot
commute.  Therefore the edges incident with $b$ in $\N$ are as claimed.
Note that if $b$ is a leaf then the edges $x_i$ are missing and we can just
omit the corresponding part of this proof.

Similar arguments show that for $k = 1, 2, \ldots, m-2$, the edges in $\MM$
around edge $d_k$ with vertices $X$ and $D_{k}$ are given clockwise around
vertex $X$ by $d_k, d_{k+1}, \ldots, d_{m-2}, a, b, d_1, d_2, \ldots, d_{k-1}$
and back to $d_k$ and around vertex $D_{k}$ by $d_k, y_1, y_2, \ldots, y_{m-1}$
and back to $d_k$.  After the dual Kauer move at $a$ the edges around $X$ are
given by $d_k, d_{k+1}, \ldots, d_{m-2}, c, a, d_1, d_2, \ldots, d_{k-1}$ and
back to $d_k$ and around vertex $D_{k}$ they remain the same. By arguments
similar to those used above there are non-zero loops of morphisms in the $T_i$
around both vertices, leading to non-zero maps in the quiver of $A'$.
Furthermore, $\Hom_{K^{\flat}(P(A_{\MM}))} (T_{d_{k}}, T_j) = 0 $ if $j$ is an edge in
any of the subtrees or unions of subtrees $G_b$, $G_c \setminus \{c \}$,
$\bigcup_{l \neq k} G_{d_{l}} \backslash \{d-l \}$, $G_{d_{k}} \backslash \{d-k,
y_1, \ldots, y_{m-1} \}$, or $\bigcup_l G_{e_{l}}$. Thus the edges around $d_k$
in the tree of $A'$ are as claimed.

Finally, consider an edge $x_k$, for $k = 1, 2, \ldots, m-1$ with vertices $B$
and $X_k$. Then in $\MM$ we have the configuration
$x_k, x_{k+1}, \ldots, x_{m-1}, b, x_1,x_2, \ldots, x_{k-1}$
around the vertex $B$ in the clockwise direction and around $X_k$ we have
the configuration $x_k, z_1, z_2, \ldots, z_{m-1}$.
The dual Kauer move on the edge $a$ does not change the configuration
around the vertices of $x_k$.  There are non-zero loops of morhpisms
$$T_{x_k}
\rightarrow T_{x_{k-1}} \rightarrow \cdots \rightarrow T_{x_{1}} \rightarrow T_b
\rightarrow T_{x_{m-1}} \rightarrow T_{x_{m-2}} \rightarrow T_{x_{k}}$$
and
$$T_{x_k} \rightarrow T_{z_{m-1}} \rightarrow T_{z_{m-2}} \rightarrow \cdots
T_{z_{1}} \rightarrow T_{x_{k}}.$$
Moreover, it is easy to check that
$\Hom_{K^{\flat}(P(A_{\MM}))} (T_{d_{k}}, T_j) = 0 $ for all $j$ in $\{ c \}$,
$G_b \setminus \{ b, x_1, x_2, \ldots, x_{m-1} \}$, $\bigcup_k G_{d_{k}}
\setminus \{ d_k \}$, $G_c \setminus \{c, w_1, w_2, \ldots, w_{m-1} \}$ (where
the edges $w_1, w_2, \ldots, w_{m-1}$ are those incident to $c$) or $G_{e_{k}}$ for
some $k$. Thus the arrows in the quiver of $A'$ are the same as in
$A_{\NN}$.

We now have considered all cases and have shown that the Brauer tree of $A'$ is
as claimed. \hfill $\square$

\begin{remark}
\begin{enumerate}[(a)]
\item The same arguments (with easy minor modifications) apply if $\MM$ is replaced
with $\mathcal{M}^*$, that is if we remove the leaf vertices (and the adjacent
edges).
\item The same arguments also apply if we consider a general `angulation' of
$P_n$, that is, with no restrictions on the sizes of the subpolygons.
This is the case since the sizes of the polygons do not play a role in the proofs.
\item Theorem~\ref{thm:flipderivedequivalent} gives a third derived equivalence in
the disk case, to add to the other two discussed in Remark~\ref{rem:twoequivalences}.
\end{enumerate}
\end{remark}

\section{Choice of graphs and counter-examples}
\label{sec:counterexamples}

In this section we describe a number of situations where there is no
derived equivalence corresponding to the dual Kauer move defined in the
previous section.
These counter-examples have motivated our choice of dual graph in the disc case.
The proofs all rely on the same principle, based on Sylvester's law of inertia,
which for the convenience of the reader we will recall here (see e.g.\
\cite{HornJohnson85}). This principle provides us with a useful
criterion for distinguishing when two finite-dimensional algebras are not
derived equivalent.

\begin{theorem}[Sylvester's law of inertia]\label{Sylvester}
Let $A$ and $B$ be symmetric real square matrices. Assume $A$ is congruent to
$B$, that is there exists a matrix $P \in \rm{GL_n} (\mathbb{R})$
such that $P A P^T = B$. Then $A$ and $B$ have the same number of strictly positive eigenvalues,
strictly negative eigenvalues, and zero eigenvalues.
\end{theorem}

We recall the following result from \cite{BocianSkowronski}.

\begin{proposition}\label{pro:conjugate}
Let $A$ and $B$ be two finite-dimensional, derived equivalent algebras with Cartan
matrices $C_A$ and $C_B$ respectively. Let $n$ denote the number of simple
modules of $A$ and $B$ (up to isomorphism). Then there exists a matrix $P \in
\rm{GL_n}(\mathbb{Z})$ such that $P C_A P^T = C_B$.
\end{proposition}

Combining Sylvester's law of inertia with Proposition~\ref{pro:conjugate} gives
a citerion which can be used to show that two finite-dimensional algebras
are not derived equivalent.

\begin{corollary}\label{EigenvalueCriterion}
Let A and B be two finite-dimensional, derived equivalent algebras with Cartan
matrices $C_A$ and $C_B$ respectively.  Then $C_A$ and $C_B$ have the same
number of strictly positive, strictly negative and zero eigenvalues.
\end{corollary}

\subsection{Dual graph of a triangulation of a polygon}

Given a triangulation $\T$ of a polygon, instead of considering the graph $\TT$
as in Section~\ref{sec:m_angulations} we could have considered the dual graph
of $\T$, which can be obtained from $\TT$ by identifying all of its boundary
vertices. We call this graph $\T^\vee$. Note that this graph is not necessarily a
Brauer tree anymore, but it is a Brauer graph (where, we think of the graph embedded in a sphere and as usual we use a local
embedding in the plane to get the cyclic ordering).
Let $a^*$ be one of the internal edges of $\T^\vee$.
We denote by $\mu^*_{a^*}(\T^\vee)$ the Brauer graph obtained by applying
a dual Kauer move to $\T^{\vee}$ at $a^*$. We shall see that the Brauer graph
algebras $A_{\T^\vee}$ and $A_{\mu^*_{a^*}(\T^\vee)}$ are not always derived equivalent.

We consider first the example in Figure~\ref{fig:DiscBoundaryIdentified}.
We will use the criterion based on Sylvester's law of inertia to show that there
is no derived equivalence in this case.

\begin{figure}[h!]
\psfragscanon \psfrag{a}{$\scriptstyle a$} \psfrag{b}{$\scriptstyle b$}
\psfrag{c}{$\scriptstyle c$} \psfrag{d}{$\scriptstyle d$}
\psfrag{e}{$\scriptstyle e$} \psfrag{f}{$\scriptstyle f$}
\psfrag{g}{$\scriptstyle g$} \psfrag{h}{$\scriptstyle h$}
\psfrag{i}{$\scriptstyle i$} \psfrag{T}{$ \T^\vee$} \psfrag{U}{$
  \mu^*_{a^*}(\T^\vee)$}
\begin{center}
\includegraphics[width=10cm]{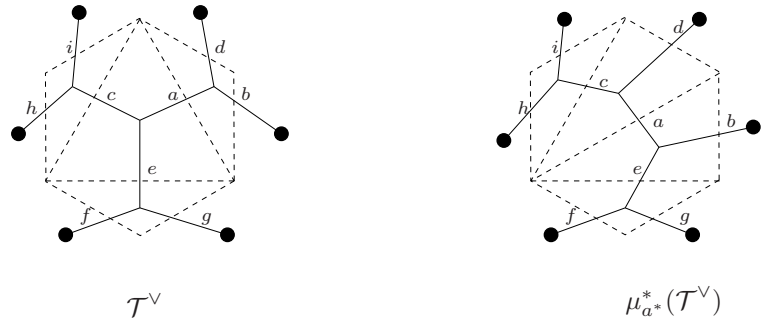}
\end{center}
\caption{In the above figures the black vertices are all identified, that is in
  each figure they represent one single vertex, namely a vertex corresponding to
  a point external to the disc.}
\label{fig:DiscBoundaryIdentified}
\end{figure}

The Cartan matrices $\mathcal{C}_{\T^\vee}$ and
$\mathcal{C}_{\mu^*_{a^*}(\T^\vee)}$ of $A_{\T^\vee}$ and
$A_{\mu^*_{a^*}(\T^\vee)}$, respectively, are given by

$$ \mathcal{C}_{\T^\vee} = \bordermatrix{ ~ & a & b & c & d & e & f & g & h & i
  \cr a& 2&1&1&1&1&0&0&0&0 \cr b&1&2&0&2&0&1&1&1&1 \cr c& 1&0&2&0&1&0&0&1&1 \cr
  d&1&2&0&2&0&1&1&1&1 \cr e&1&0&1&0&2&1&1&0&0 \cr f&0&1&0&1&1&2&2&1&1 \cr
  g&0&1&0&1&1&2&2&1&1 \cr h&0&1&1&1&0&1&1&2&2 \cr i&0&1&1&1&0&1&1&2&2 \cr }
\mbox{ and } \mathcal{C}_{\mu^*_{a^*}(\T^\vee)} = \bordermatrix{ ~ & a & b & c &
  d & e & f & g & h & i \cr a& 2&1&1&1&1&0&0&0&0 \cr b&1&2&0&1&1&1&1&1&1 \cr c&
  1&0&2&1&0&0&0&1&1 \cr d&1&1&1&2&0&1&1&1&1 \cr e&1&1&0&0&2&1&1&0&0 \cr
  f&0&1&0&1&1&2&2&1&1 \cr g&0&1&0&1&1&2&2&1&1 \cr h&0&1&1&1&0&1&1&2&2 \cr
  i&0&1&1&1&0&1&1&2&2 \cr }$$

\bigskip

The characteristic polynomial of $\mathcal{C}_{\T^\vee}$ is
$$ P_{\T^\vee}(x) = x^9-18x^8+111x^7-288x^6+270x^5,$$ and
$\mathcal{C}_{\T^\vee}$ has 5 eigenvalues equal to zero.
The characteristic polynomial of $\mathcal{C}_{\mu^*_{a^*}(\T^\vee)}$ is
$$ P_{\mu^*_{a^*}(\T^\vee)} (x) = x^9-18x^8+113x^7-316x^6+395x^5-180x^4$$ and
$\mathcal{C}_{\mu^*_{a^*}(\T^\vee)}$ has 4 eigenvalues equal to zero. It
follows from Corollary~\ref{EigenvalueCriterion} that the algebras $A_{\T^\vee}$
and $A_{\mu^*_{a^*}(\T^\vee)}$ are not derived equivalent.

\subsection{Punctured disc case}

Even without making the boundary identification in the dual graph as in the
previous example, in the punctured disc case we have an example of
a triangulation $\T$ and an edge $a$, where the Brauer
graph algebras corresponding to the graphs $\TT$ and $\mu^*_{a^*}(\TT)$
(as defined in Section~\ref{sec:m_angulations}) are not derived equivalent.

\begin{figure}[h!]
\psfragscanon \psfrag{a}{$\scriptstyle a$} \psfrag{b}{$\scriptstyle b$}
\psfrag{c}{$\scriptstyle c$} \psfrag{d}{$\scriptstyle d$}
\psfrag{e}{$\scriptstyle e$} \psfrag{f}{$\scriptstyle f$}
\psfrag{g}{$\scriptstyle g$} \psfrag{h}{$\scriptstyle h$} \psfrag{T}{$\TT$}
\psfrag{U}{$ \mu^*_{a^*}(\TT)$}
\begin{center}
\includegraphics[width=10cm]{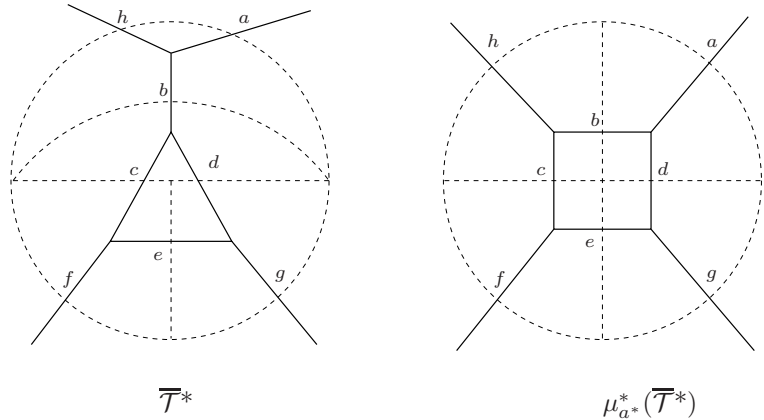}
\end{center}
\caption{Triangulation of a punctured disc and dual graph}
\label{fig:PuncturedDiscCE}
\end{figure}

In Figure~\ref{fig:PuncturedDiscCE} we consider the dual graph of a triangulation
of a disc with $4$ boundary vertices and one puncture. Note that the boundary
vertices are not identified in this set-up.

The Cartan matrices $\mathcal{C}_{\T^\vee}$ and
$\mathcal{C}_{\mu^*_{a^*}(\T^\vee)}$ of $A_{\T^\vee}$ and
$A_{\mu^*_{a^*}(\T^\vee)}$, respectively, are given by

$$ \mathcal{C}_{\T^\vee} = \bordermatrix{ ~ & a & b & c & d & e & f & g & h \cr
  a& 2&1&0&0&0&0&0&1 \cr b&1&2&1&1&0&0&0&1 \cr c& 0&1&2&1&1&1&0&0 \cr
  d&0&1&1&2&1&0&1&0 \cr e&0&0&1&1&2&1&1&0 \cr f&0&0&1&0&1&2&0&0 \cr
  g&0&0&0&1&1&0&2&0 \cr h&1&1&0&0&0&0&0&2 \cr } \mbox{ and }
\mathcal{C}_{\mu^*_{a^*}(\T^\vee)} = \bordermatrix{ ~ & a & b & c & d & e & f &
  g & h \cr a& 2&1&0&1&0&0&0&0 \cr b&1&2&1&1&0&0&0&1 \cr c& 0&1&2&0&1&1&0&1 \cr
  d&1&1&0&2&1&0&1&0 \cr e&0&0&1&1&2&1&1&0 \cr f&0&0&1&0&1&2&0&0 \cr
  g&0&0&0&1&1&0&2&0 \cr h&0&1&1&0&0&0&0&2 \cr }$$

\bigskip

Then the determinant of $\mathcal{C}_{\T^\vee}$ is 4 and the determinant of $\mathcal{C}_{\mu^*_{a^*}(\T^\vee)}$ is 0
and thus by proposition \ref{pro:conjugate} the algebras $A_{\T^\vee}$
and $A_{\mu^*_{a^*}(\T^\vee)}$ are not derived equivalent.

\end{document}